\documentclass[12pt]{article}
\usepackage{amssymb}
\usepackage{amsfonts}
\usepackage{german}
\newtheorem{theorem}{Satz}

\newtheorem{corollary}[theorem]{Folgerung}
\newtheorem{remark}[theorem]{Bemerkung}

\newtheorem{lemma}[theorem]{Lemma}

\newcommand{\ay}{{\mathfrak{a}}}
\newcommand{\py}{{\mathfrak{p}}}

\newcommand{\my}{{\mathfrak{m}}}
\newcommand{\qy}{{\mathfrak{q}}}

\newcommand{\by}{{\mathfrak{b}}}

\newcommand{\Gy}{{\mathfrak{G}}}

\newenvironment{proof}{{\it Beweis}. $\;\;$}{\hspace*{\fill} $\Box$}

\begin{document}

\title{\Large\bf "Uber die maximalen Ideale des Quotientenringes $R_{\Gy}$}

\author{Helmut Z"oschinger\\ 
Mathematisches Institut der Universit"at M"unchen\\
Theresienstr. 39, D-80333 M"unchen, Germany\\
E-mail: zoeschinger@mathematik.uni-muenchen.de}

\date{}
\maketitle

\vspace{1cm}

\begin{center}  
{\bf Abstract}
\end{center}

Let $R$ be a commutative Noetherian local ring, $\Gy$ a Gabriel topology on $R$, and $\Gy^\ast$ the set of all maximal elements of Spec($R)\backslash \Gy$.
We determine all simple $\Gy$-torsion free $R$-modules $M$, as well as all simple $\Gy$-divisible
Artinian $R$-modules $N$. A central role is played by the set $\Gy^\ast$: Its elements correspond exactly to the
maximal ideals of the quotient ring $R_\Gy$, if $\Gy$ is perfect.

\vspace{0.5cm}


\vspace{0.5cm}



\section{Die Menge $\Gy^\ast$ aller maximalen Elemente\\ von Spec$(R)\backslash \Gy$}
 
Stets sei $R$ ein kommutativer, noetherscher, lokaler Ring und $\Gy$ eine Gabriel-Topologie auf $R$.
Ein $R$-Modul $M$ hei"st bekanntlich $\Gy$-torsionsfrei, wenn Ass$(M) \cap \Gy =\emptyset$ ist, {\bf einfach-$\Gy$-torsionsfrei}, wenn $M$ $\Gy$-torsionsfrei $\neq 0$ ist
und aus $0\neq U \subsetneq M$ stets folgt, da"s $M/U$ nicht $\Gy$-torsionsfrei ist.
Dual hei"st $N\;\Gy$-teilbar, wenn Koass$(N) \cap \Gy =\emptyset$ ist, {\bf einfach-$\Gy$-teilbar}, wenn $N \;\Gy$-teilbar $\neq 0$ ist und aus $0 \neq V \subsetneq N$ stets folgt, da"s $V$ nicht $\Gy$-teilbar ist.
Im Spezialfall $\Gy_1=\{\ay \subset R |\ay$ enth"alt einen NNT von $R\}$ ist der letzte Begriff schon
ausf"uhrlich untersucht worden, von Matlis in \cite{4} im Fall dim($R)\leq 1$, von Facchini in \cite{2} f"ur beliebige Integrit"atsringe, 
endlich f"ur beide Begriffe in \cite{12}.

\newpage

{\bf Beispiele}

\begin{enumerate}
\item[(1)]
Ist $\Gy_1=\{\ay \subset R|\; \ay $ regul"ar$\}$ wie eben, so ist $\Gy_1^\ast$ die Menge aller maximalen Elemente von Ass$(R)$.

\item[(2)]
Ist $\qy$ ein Primideal von $R$ und $\Gy=\{\ay \subset R|\; \ay \nsubseteq \qy\}$ so ist $\Gy^\ast = \{\qy\}$.

\item[(3)]
Ist $n\geq 0$ und $\Gy = \{R\}\; \cup\; \{\ay \subsetneq R| \dim (R/\ay) \leq n\}$, so ist $\Gy^\ast = \{\py \in $ Spec$(R)| \dim(R/\py)=n+1\}$.
\end{enumerate}

\begin{lemma}

F"ur jeden einfach-$\Gy$-torsionsfreien $R$-Modul $M$ gilt:
\begin{enumerate}
\item[(a)]
$M $ ist prim.
\item[(b)]
$M$ ist uniform.
\item[(c)]
Ann$_R(M) \in \Gy^\ast$.
\end{enumerate}
\end{lemma}

\begin{proof}
\begin{enumerate}
\item[(a)]
Sogar jeder Endomorphismus $0 \neq f: M \to M$ ist injektiv, denn $M/{\rm Ke}f \cong $ Bi$f$ ist wieder $\Gy$-torsionsfrei, also Ke$f =0.$
\item[(b)]
Sogar f"ur jeden Untermodul $0 \neq U \subset M$ ist $M/U\;\Gy$-torsion, denn mit $U_1/U=\{\bar{x} \in M/U|\; \ay \bar{x} =0$ f"ur ein $\ay \in \Gy\}$ folgt $M/U_1\;\Gy$-torsionsfrei, also $U_1=M$.
Damit ist auch $U$ einfach-$\Gy$-torsionsfrei, und aus $V \subset M,\;V \cap U=0$ folgt $V\;\Gy$-torsion, $V=0$.
\item[(c)]
Nach \cite[Lemma 1.1]{12} ist $\py = $ Ann$_R(M) \in $ Spec$(R)$ und Ass$(M) = \{\py\}$, also $\py \notin \Gy$, und aus $\py \subsetneq \qy$ folgt $R/\qy\;\Gy$-torsion, $\qy \in \Gy$.
\end{enumerate}
\end{proof}  

Wir zeigen im folgenden (1.5), da"s diese drei Bedingungen sogar "aquivalent sind mit ''$M$ einfach-$\Gy$-torsionsfrei'', und in mehreren F"allen (c) allein gen"ugt.

\begin{lemma}
Ist $\py \in \Gy^\ast$ und $N$ ein $R$-Modul mit $\py N=0$, so gilt:
\begin{enumerate}
\item[(a)]
$N\;\Gy$-torsionsfrei $\;\Longleftrightarrow\;N$ als $R/\py$-Modul torsionsfrei.
\item[(b)]
$N\;\Gy$-teilbar $\;\Longleftrightarrow\; N$ als $R/\py$-Modul teilbar.
\end{enumerate}
\end{lemma}

\pagebreak

\begin{proof}
\begin{enumerate}
\item[(a)]
$''\Rightarrow\;''$ \quad
F"ur jedes $0\neq \bar{s} \in R/\py$ ist $s \notin \py$, also nach Voraussetzung $\sqrt{\py +(s)} \in \Gy,\;\py +(s) \in \Gy,\;N[\py +(s)]=0,\;N[s]=0,\;N[\bar{s}]=0$.

$''\Leftarrow\;''$ \quad
F"ur jedes $\ay \in \Gy$ ist nach Voraussetzung $\ay \nsubseteq \py$, mit $s_0 \in \ay \backslash \py$ also $N[\bar{s}_0]=0,\;N[s_0]=0,\;N[\ay]=0.$
\item[(b)]
Entsprechend.
\end{enumerate}
\end{proof}

\begin{corollary}
Ist $\py \in \Gy^\ast$ und $N$ ein $R$-Modul mit $\py N=0$, so gilt:
\begin{enumerate}
\item[(a)]
$N$  einfach-$\Gy$-torsionsfrei $\; \Longleftrightarrow  \;N$ als $R/\py$-Modul einfach-torsionsfrei.
\item[(b)]
$N$ einfach-$\Gy$-teilbar $\;\Longleftrightarrow\;N$ als $R/\py$-Modul einfach-teilbar.
\end{enumerate}
\end{corollary}

\begin{theorem}
Sei $\py$ ein Primideal und $\kappa(\py)=R_\py/\py R_\py$.
Dann sind "aquivalent:
\begin{enumerate}
\item[(i)]
$\kappa(\py)$ ist als $R$-Modul einfach-$\Gy$-torsionsfrei.
\item[(ii)]
$\kappa(\py)$ ist als $R$-Modul einfach-$\Gy$-teilbar.
\item[(iii)]
$\py \in \Gy^\ast$.
\end{enumerate}
\end{theorem}

\begin{proof}
Mit $N=\kappa(\py)$ ist Ann$_R(N)=\py$ und $N$ als $R/\py$-Modul der Quotientenk"orper von $R/\py$,
also $N$ als $R/\py$-Modul einfach-torsionsfrei und einfach-teilbar, so da"s mit (1.3) sofort folgt $(iii) \to (i)$ und $(iii)  \to (ii)$, mit (1.1) aber auch $(i) \to (iii)$.

Bleibt $(ii) \to (iii)$ zu zeigen, worin $\py \notin \Gy$ klar ist wegen $\py N \neq N$.\\
1.Fall: $\dim(R/\py)=0$. Dann gibt es  kein  echt gr"o"seres Primideal, so da"s $\py \in \Gy^\ast$ ist.\\
2.Fall: $\dim(R/\py)=1$. Dann ist $0  \neq R/\py \subsetneq \kappa(\py)$, also nach Voraussetzung
$ \Gy \neq \{R\},\;\my \in \Gy,\;\py \in \Gy^\ast$.\\
3.Fall: $\dim(R/\py)\geq 2.$ Dann ist wie im 2.Fall  $\my \in \Gy$, und f"ur jedes Primideal $\py \subsetneq \qy \subsetneq \my$ m"ussen wir jetzt $\qy \in \Gy$ zeigen:\\
o.B.d.A. sei  $h(\qy/\py)=1.$ "Uber dem Integrit"atsring $\bar{R} =R/\py$ gilt f"ur die Lokalisierung $U=\bar{R}_{\bar{\qy}}$ nach \cite[Folgerung 4.7]{10} Koass$_{\bar{R}}(U) =\{0,\bar{\qy}\},$
d.h. Koass$_R(U) =\{\py,\qy\}$.
Wegen $R/\py \subsetneq U \subsetneq \kappa(\py)$ ist nach Voraussetzung $_RU$ nicht $\Gy$-teilbar, so da"s aus Koass$_R(U) \cap \Gy \neq \emptyset$ folgt $\qy \in \Gy$.
\end{proof}

\begin{theorem}
F"ur einen $R$-Modul $M$ sind "aquivalent:
\begin{enumerate}
\item[(i)]
$M$ ist einfach-$\Gy$-torsionsfrei.
\item[(ii)]
$M$ ist prim, uniform und Ann$_R(M) \in  \Gy^\ast$.
\item[(iii)]
Es gibt ein $\py \in \Gy^\ast$ mit $0\neq M  \hookrightarrow \kappa(\py)$.
\end{enumerate}
\end{theorem}

\begin{proof}
Bei $(iii) \to (i)$ ist $\kappa(\py)$ nach (1.4)  einfach-$\Gy$-torsionsfrei, also auch  der Untermodul $M$, und $(i) \to (ii)$  ist (1.1).

Blebt $(ii) \to (iii)$ zu zeigen: Weil $M$ prim, also $\py =$ Ann$_R(M)$ ein Primideal und Ass$(M)=\{\py\}$ ist \cite[Lemma 1.1]{12}, folgt nach Voraussetzung $\py \in \Gy^\ast$ und $M \hookrightarrow E(R/\py)[\py]\cong \kappa(\py)$.
\end{proof}

Mit $(ii)$ erh"alt man sofort folgendes:

{\bf Beispiel 1}\quad
Ein zyklischer $R$-Modul $M$ ist genau dann einfach-$\Gy$-torsionsfrei,  wenn Ann$_R(M) \in \Gy^\ast$ ist.

Eine weitere Anwendung von $(ii)$ ergibt sich aus folgender  Variante von \cite[Satz 1.2]{12}:

\begin{lemma}
Sei $S$ eine multiplikative Teilmenge von $R$ und $M=R_S$. Dann sind "aquivalent:
\begin{enumerate}
\item[(i)]
$M$ ist als $R$-Modul prim.
\item[(ii)]
Ann$_R(M)$ ist ein Primideal.
\item[(iii)]
$R_S$ ist ein Integrit"atsring.
\item[(iv)]
Es gibt ein $\py \in$ Spec$(R)$ mit $0 \neq M \hookrightarrow \kappa(\py).$
\end{enumerate}
\end{lemma}

\begin{proof}
Klar ist $(i) \to (ii)$, und bei $(ii) \to (iii)$ folgt mit $\py:=$ Ann$_R(M)$, da"s $\py \cap S =\emptyset$ ist [denn ein $s_0 \in \py$ w"urde $s_0M=M$, also $M=0$ implizieren],
also $\py R_S=0\in$ Spec$(R_S)$.

Bei $(iii) \to (i)$ ist jeder $R$-Endomorphismus $0  \neq f: M \to M$ auch $R_S$-linear, also injektiv.

Klar ist $(iv) \to (ii)$ wegen Ann$_R(M)=\py$, und bei $(iii) \to (iv)$ ist $M$ auch als $R$-Modul  uniform, so da"s mit $\py:=$ Ann$_R(M)$ folgt $M \hookrightarrow E(R/\py)[\py] \cong \kappa(\py).$
\end{proof}

{\bf Beispiel 2}\quad
Sei $\Gy$ eine Gabriel-Topologie auf $R$. Genau dann ist der $R$-Modul $M=R_S$ einfach-$\Gy$-torsionsfrei, wenn Ann$_R(M) \in \Gy^\ast$ ist.

Falls ein $R$-Modul $N$ einfach-$\Gy$-teilbar und {\bf reflexiv} ist, gibt es nach (1.8) auch einen Epimorphismus $\kappa(\py) \twoheadrightarrow N$ mit $\py \in \Gy^\ast$:

\begin{lemma}
Sei $E$ die injektive H"ulle von $R/\my$ und $N^\circ = $ Hom$_R(N,E)$ das Matlis-Duale von $N$. Dann gilt:
\begin{enumerate}
\item[(a)]
$N^\circ$ einfach-$\Gy$-torsionsfrei $\;\Longrightarrow\;N$ einfach-$\Gy$-teilbar.
\item[(b)]
$N^\circ$ einfach-$\Gy$-teilbar $\; \Longrightarrow\;N$ einfach-$\Gy$-torsionsfrei.
\item[(c)]
Ist $N$ reflexiv, gilt beide Male die Umkehrung.
\end{enumerate}
\end{lemma}

\begin{proof} \quad\\
Vorbemerkung:
Nach \cite[Hilfssatz 2.2]{11} gilt f"ur jeden $R$-Modul $X$ und jedes Ideal $\ay \subset R$, da"s $X^\circ[\ay] = $ Ann$_{X^\circ}(\ay X)$ und 
$X^\circ /\ay X^\circ \cong (X[\ay])^\circ$ ist, also
$$ N^\circ \;\Gy\mbox{-torsionsfrei}\;\Longleftrightarrow\;N\;\Gy\mbox{-teilbar},$$
$$N^\circ \;\Gy\mbox{-teilbar}\;\Longleftrightarrow\;N\;\Gy\mbox{-torsionsfrei}.$$
\begin{enumerate}
\item[(a)]
F"ur jeden Untermodul $0 \neq V \subsetneq N$ folgt aus $0 \neq $ Ann$_{N^\circ}(V)\subsetneq N^\circ$, da"s $N^\circ /$Ann$_{N^\circ}(V) \cong V^\circ$ nach Voraussetzung {\bf nicht} $\Gy$-torsionsfrei ist, also $V$ nicht $\Gy$-teilbar.
\item[(b)]
Entsprechend folgt aus $0 \neq V \subsetneq N,\;0 \neq $ Ann$_{N^\circ}(V) \subsetneq N^\circ$, da"s Ann$_{N^\circ}(V) \cong (N/V)^\circ$ {\bf nicht} $\Gy$-teilbar ist, also $N/V$ nicht $\Gy$-torsionsfrei.
\item[(c)]
Klar wegen $N \cong N^{\circ\circ}$.
\end{enumerate}
\end{proof}

\begin{corollary}
Ist $N$ einfach-$\Gy$-teilbar und reflexiv, so gibt es einen Epimorphismus $\kappa(\py) \twoheadrightarrow N$ mit $\py \in \Gy^\ast$.
\end{corollary}

\begin{proof}
Weil $N^\circ$ nach (c) einfach-$\Gy$-torsionsfrei ist, gibt es nach (1.5) einen Monomorphismus $N^\circ \hookrightarrow \kappa(\py)$ mit $\py \in \Gy^\ast$, also einen Epimorphismus $\kappa(\py)^\circ \twoheadrightarrow N.$
Weil $\kappa(\py) \;\Gy$-torsionsfrei, also $\kappa(\py)^\circ \cong \kappa(\py) ^{(I)}\;\Gy$-teilbar ist, folgt die Behauptung.
\end{proof}

\begin{remark}
Ohne die Reflexivit"at von $N$ gelten weder die Umkehrungen in (1.7) noch die Aussage in (1.8):
Nach \cite[p.6]{12} gilt f"ur jeden Integrit"atsring $R$, bei dem $\hat{R}$ einen NT $\neq 0$ besitzt, da"s $N=R$ einfach-torsionsfrei,
aber $N^\circ \cong E$ {\bf nicht} einfach-teilbar ist.

Nach \cite[p.9]{12} gibt es einen Ring $R$ mit einem einfach-teilbaren $R$-Modul $N$, so da"s $\py = $ Ann$_R(N) \notin $ Ass$(R)$ ist.
Damit kann $N^\circ$ {\bf nicht} einfach-torsionsfrei sein, aber auch {\bf kein} Epimorphismus $\kappa(\qy) \twoheadrightarrow N$ existieren 
mit $\kappa(\qy)$ einfach-teilbar [denn dann w"are $\qy$ nach (1.4) ein maximales Element von Ass$(R)$ im Widerspruch zu $\qy \subset \py \subsetneq \py_1 \in $ Ass$(R)$].
\end{remark}

\section{Stark-(ko)prime Moduln}
\setcounter{theorem}{0}

Ein einfach-$\Gy$-torsionsfreier $R$-Modul $M$ hat die Eigenschaft, da"s jeder Endomorphismus $0 \neq f: M \to M$ injektiv ist (siehe die Beweise von (1.1 a) oder (1.6 i)),
was wir mit {\bf stark-prim} bezeichnen wollen.
Dual hei"se ein $R$-Modul $N \neq 0$ {\bf stark-koprim}, wenn jeder Endomorphismus $0\neq f: N \to N$ surjektiv ist.
Offenbar hat man die Implikationen
$$ \mbox{einfach-}\Gy\mbox{-torsionsfrei}\;\Longrightarrow \;\mbox{stark-prim}\;\Longrightarrow\;\mbox{prim},$$
$$\mbox{einfach-}\Gy\mbox{-teilbar}\;\Longrightarrow\;\mbox{stark-koprim}\;\Longrightarrow\;\mbox{koprim},$$
und im allgemeinen gelten keine Umkehrungen.

{\bf Beispiel 1}\quad 
$\kappa(\py)=R_\py/\py\,R_\py$ ist als $R$-Modul sowohl stark-prim als auch stark-koprim.

\begin{proof}
Jeder $R$-Endomorphismus $f: \kappa(\py) \to \kappa(\py)$ ist auch $R/\py$-linear, also --- weil $\kappa(\py)$ der Quotientenk"orper von $R/\py$ ist --- Null oder bijektiv.
\end{proof}

\begin{lemma}
$$M \mbox{ prim und uniform} \;\Longrightarrow\; M\;\mbox{stark-prim}.$$
\end{lemma}

\begin{proof}
$\py = $ Ann$_R(M)$ ist nach \cite[Lemma 1.1]{12} ein Primideal mit Ass$(M) = \{\py\}$, und aus $M \subset E(R/\py)[\py] \cong \kappa(\py)$ folgt, da"s jeder $R$-Endomorphismus $0 \neq f: M \to M$ injektiv ist.
\end{proof}

{\bf Beispiel 2}\quad
F"ur jeden halbartinschen $R$-Modul $N$ gilt bekanntlich End$_R(N)=$ End$_{\hat{R}}(N)$, so da"s folgt:
Genau dann ist $N$ als $R$-Modul stark-prim (stark-koprim), wenn $N$ als $\hat{R}$-Modul stark-prim (stark-koprim) ist.

{\bf Beispiel 3}\quad
\begin{enumerate}
\item[(a)]
Ein zyklischer $R$-Modul $M$ ist genau dann stark-prim, wenn Ann$_R(M) \in $ Spec$(R)$ ist.
\item[(b)]
Ein kozyklischer $R$-Modul $N$ ist genau dann stark-koprim, wenn Ann$_{\hat{R}}(N) \in $ Spec$(\hat{R})$ ist.
\end{enumerate}

\begin{proof}
\begin{enumerate}
\item[(a)]
Mit $\py=$ Ann$_R(M)$ ist $M \cong R/\py$ prim und uniform, so da"s (2.1) die Behauptung liefert.

\item[(b)]
Nach Beispiel 2 kann man $R=\hat{R}$ annehmen, so da"s $N \subset E$ reflexiv ist, also mit (a) die Behauptung folgt.
\end{enumerate}
\end{proof}

\begin{lemma}
\begin{enumerate}
\item[(a)]
$\;N^\circ$ stark-prim $\;\Longrightarrow\;N$ stark-koprim.
\item[(b)]
$\;N^\circ$ stark-koprim $\;\Longrightarrow\;N$ stark-prim.
\item[(c)]
Ist $N$ reflexiv, gilt beide Male die Umkehrung.
\end{enumerate}
\end{lemma}

\begin{proof}
Bei (a) folgt aus $0 \neq f: N \to N$, da"s $0\neq f^\circ: N^\circ \to N^\circ$ ist,
also nach Voraussetzung $f^\circ$ injektiv, $f$ surjektiv.
Entsprechend (b) und (c) (siehe auch \cite[Lemma 3.2]{12}).
\end{proof}

\begin{remark}
Ohne ''reflexiv'' gelten keine Umkehrungen:
Ist $R$ ein Integrit"atsring mit Quotientenk"orper $K$ und $\dim(R) \geq 2,$ so ist $K$ nach Beispiel 1 stark-prim und stark-koprim, aber $K^\circ \cong K^{(I)}$ mit $|I| \geq 2$ direkt zerlegbar.
\end{remark}

\begin{lemma}
$N=E(R/\qy)$ als $R$-Modul stark-koprim $\;\Longleftrightarrow\;\, R_\qy$ als Ring analytisch irreduzibel.
\end{lemma}

\begin{proof}
Auch jetzt ist End$_R(N) = $ End$_{R_\qy}(N)$ [denn alle $s \in R\backslash \qy$ operieren auf $N$ bijektiv via
$\frac{r}{s} \cdot x =ry$ mit $x=sy$]
und $N$ ist als $R_\qy$-Modul die injektive H"ulle von $\kappa(\qy)$, so da"s Beispiel 3 die Behauptung liefert.
\end{proof}

\begin{remark}
Wir wissen nicht, wann $N=E(R/\qy)$ als $R$-Modul einfach-teilbar ist (siehe \cite[p.7]{12}).
Ist aber $\Gy$ eine Gabriel-Topologie auf $R$ und $\qy \notin \Gy$, folgt  mit (1.4) sofort:
Genau dann ist $N=E(R/\qy)$ als $R$-Modul einfach-$\Gy$-teilbar, wenn $R_\qy$ ein K"orper ist und $\qy \in \Gy^\ast$.
\end{remark}

\section{Einfach-$\Gy$-teilbare Moduln}
\setcounter{theorem}{0}

Die explizite Beschreibung aller einfach-$\Gy$-torsionsfreien $R$-Moduln $M$ in (1.5) l"a"st sich nicht ohne weiteres f"ur einfach-$\Gy$-teilbare $R$-Moduln $N$ dualisieren:
$N$ mu"s nicht kouniform sein, es kann Ann$_R(N) \notin \Gy^\ast$ sein, und es mu"s keinen Epimorphismus $\kappa(\py) \twoheadrightarrow N$ geben mit $\py \in \Gy^\ast$.
Nur wenn $N$ komplementiert ist (3.1) oder sogar artinsch (3.3), erh"alt man teilweise duale Resultate.

\begin{lemma}
Ist $N$ einfach-$\Gy$-teilbar und {\bf komplementiert}, so gilt:
\begin{enumerate}
\item[(a)]
$\;N$ ist kouniform.
\item[(b)]
Zu jedem artinschen Untermodul $V\subsetneq N$ gibt es ein $\ay_0 \in \Gy$ mit $\ay_0V=0$ und $N/V\twoheadrightarrow N$.
\item[(c)]
Ist sogar {\bf jeder} Untermodul von $N$ komplementiert, sind alle Faktormoduln $\neq 0$ wieder einfach-$\Gy$-teilbar.
\end{enumerate}
\end{lemma}

\begin{proof}
\begin{enumerate}
\item[(a)]
Jeder Untermodul $V\subsetneq N$ ist klein in $N$, denn f"ur ein Komplement $W$ von $V$ in $N$ gilt $\ay W+V=N,\;\ay W=W$ f"ur alle $\ay \in \Gy$, d.h. $W$ ist $\Gy$-teilbar, $W=N$.
\item[(b)]
Die Menge $\{\ay V|\;\ay \in \Gy\}$ hat ein minimales Element $\ay_0 V$, f"ur alle $\by \in \Gy$ gilt dann in $\by\ay_0V \subset \ay_0V$ Gleichheit, so da"s $\ay_0V\;\Gy$-teilbar, also nach Voraussetzung Null ist.
Nach (a) ist aber $N=\ay_0N$ kouniform, $N=r_0N$ f"ur ein $r_0 \in \ay_0,\;V \subset N[\ay_0]\subset N[r_0]$, also $N/V \twoheadrightarrow N/N[r_0] \cong N$.
\item[(c)]
Sei $N/V \neq 0.$ Klar ist $N/V\;\Gy$-teilbar, und g"abe es einen $\Gy$-teilbaren Untermodul $0\neq W/V \subsetneq N/V$,
folgte f"ur ein Komplement $W'$ von $V$ in $W$, da"s auch $W'\;\Gy$-teilbar w"are entgegen der Voraussetzung.
\end{enumerate}
\end{proof}

\begin{remark}
F"ur jeden einfach-$\Gy$-torsionsfreien $R$-Modul $M$ gilt nat"urlich ohne Zusatzbedingungen:
Ist $0 \neq U \subset M$ und $M/U$ endlich erzeugt, so gibt es ein $\ay_0 \in \Gy$ mit $\ay_0 \cdot M/U =0$ und $M\hookrightarrow U$.
\end{remark}

\begin{corollary}
Ist $N$ einfach-$\Gy$-teilbar und {\bf artinsch}, so ist $N$ epi-"aquivalent zu genau einem einfach-$\Gy$-teilbaren Untermodul $N_1 \subset E$.
\end{corollary}

\begin{proof}
Bekanntlich gibt es Untermoduln $V \subsetneq N$ und $N_1 \subset E$ mit $N/V \cong N_1$, nach (c) ist $N_1$ wieder einfach-$\Gy$-teilbar, und nach (b) $N_1\twoheadrightarrow N \twoheadrightarrow N_1.$
Die Eindeutigkeit folgt wie in \cite[Lemma 4.1]{12}.
\end{proof}

\begin{lemma}
F"ur einen Untermodul $0\neq N \subset E$ sind "aquivalent:
\begin{enumerate}
\item[(i)]
$N$ ist $\Gy$-teilbar.
\item[(ii)]
F"ur jeden Primdivisor $P$ von Ann$_{\hat{R}}(N)$ in $\hat{R}$ gilt $P \cap R \notin \Gy$.
\end{enumerate}
\end{lemma}

\begin{proof}
Ein beliebiger $R$-Modul $X$ ist genau dann $\Gy$-teilbar, wenn Koass$_R(X) \cap \Gy =\emptyset $ ist.
In unserem Fall is Koass$_R(N)=\{P \cap R\,|\,P \in $ Koass$_{\hat{R}}(N)\}$ \cite[Satz 2.9]{10} und Koass$_{\hat{R}}(N)=$ Ass$_{\hat{R}}(\hat{R} /$Ann$_{\hat{R}}(N))$ (Matlis-Dualit"at), so da"s die Behauptung folgt.
\end{proof}

Der folgende Satz, zusammen mit (3.3), beschreibt alle einfach-$\Gy$-teilbaren, artinschen $R$-Moduln $N$:

\begin{theorem}
F"ur einen Untermodul $0\neq N \subset E$ sind "aquivalent:
\begin{enumerate}
\item[(i)]
$\;N$ ist einfach-$\Gy$-teilbar.
\item[(ii)]
$\;$ Ann$_{\hat{R}}(N)$ ist ein maximales Element in der Menge
$$\{ P \in \mbox{ Spec}(\hat{R})\,|\;P \cap R \notin \Gy\}.$$
Gen"ugt $R\subset \hat{R}$ zus"atzlich der Bedingung ''going up'', so ist (ii) weiter "aquivalent mit
\item[(iii)]
$\;N$ ist stark-koprim und Ann$_R(N) \in \Gy^\ast.$
\end{enumerate}
\end{theorem}

\begin{proof}\quad
$(i) \to (ii)\;\;N$ ist stark-koprim, also nach Beispiel 3 im letzten Abschnitt Ann$_{\hat{R}}(N) \in $ Spec$(\hat{R})$ und nat"urlich Ann$_{\hat{R}}(N) \cap R = $ Ann$_R(N) \notin \Gy.$

Zur Maximalit"at sei jetzt $P \in $ Spec$(\hat{R})$ mit Ann$_{\hat{R}}(N) \subset P$ und $P \cap R \notin \Gy$.
Dann ist $V=$ Ann$_E(P)$ als $R$-Modul nach (3.4) $\Gy$-teilbar, denn f"ur jeden Primdivisor $Q$ von Ann$_{\hat{R}}(V)$ gilt $P=Q$, also
$Q \cap R \notin \Gy$. In $V \subset N$ folgt nach Voraussetzung $V=N$, also Ann$_{\hat{R}}(N)=P.$

$(ii) \to (i)\;$ Wieder nach (3.4) ist $N$ als $R$-Modul $\Gy$-teilbar, denn f"ur jeden Primdivisor $P$ von Ann$_{\hat{R}}(N)$ gilt nach Voraussetzung $P=$ Ann$_{\hat{R}}(N),\;P \cap R \notin \Gy$.

$N$ ist sogar einfach-$\Gy$-teilbar, denn aus $0\neq V \subset N$ und $V\;\Gy$-teilbar folgt mit irgendeinem Primdivisor $Q_0$ von Ann$_{\hat{R}}(V)$, da"s nach (3.4) $Q_0 \cap R \notin \Gy$ ist, in $Q_0\supset $ Ann$_{\hat{R}}(V) \supset $ Ann$_{\hat{R}}(N)$ also nach Voraussetzung Gleichheit gilt, insbesondere $V=N$.

Erf"ulle jetzt $R \subset \hat{R}$ die Bedingung ''going up'':

$(ii) \to (iii)\;$ Nat"urlich ist Ann$_R(N) \in $ Spec$(R)\backslash \Gy.$
F"ur jedes Ann$_R(N) \subsetneq \qy \in $ Spec$(R)$ gibt es aber wegen ''going up'' ein Ann$_{\hat{R}}(N) \subsetneq Q$ mit $Q \cap R = \qy$ und es folgt $\qy \in \Gy$.

$(iii) \to (ii)\;$ Klar ist Ann$_{\hat{R}}(N) \cap R \notin \Gy.$
F"ur jedes Primideal Ann$_{\hat{R}}(N) \subsetneq Q$ folgt aber wegen (INC) Ann$_R(N)\subsetneq Q \cap R$, also $Q \cap R \in \Gy$.
\end{proof}

\begin{remark}\quad
Ohne Zusatzbedingungen an $R \subset \hat{R}$ sind die Punkte $(ii)$ und $(iii)$ voneinander unabh"angig, selbst bei $\Gy =\{\ay \subset R|\,\ay$ ist regul"ar$\}$:

F"ur $(ii) \nrightarrow (iii)$ sei $R$ ein sockelfreier Ring mit einem Primideal $P \subset \hat{R}$, so da"s $\dim(\hat{R} /P)=1,\;P \cap R$ nicht regul"ar und $P \cap R \notin $ Ass$(R)$ ist (siehe \cite[p.9]{12}).
Dann ist $N:= $ Ann$_E(P)$ einfach-teilbar, aber  Ann$_R(N) \notin \Gy^\ast$.
F"ur $(iii) \nrightarrow (ii)$ sei $R$ analytisch irreduzibel und $0\neq P \subset \hat{R}$ ein Primideal mit $P \cap R =0.$
Dann erf"ullt $N:= E$ die Bedingung $(iii)$, nicht aber $(ii)$.
\end{remark}

\section{Spezielle Gabriel-Topologien}
\setcounter{theorem}{0}

Die in \cite{12} betrachteten Gabriel-Topologien $\Gy_1=\{\ay \subset R\,|\,\ay $ ist regul"ar$\}$ und $\Gy_2=\{\ay \subset R\,| \,R/\ay$ ist artinsch$\}$ haben
folgende nat"urliche Verallgemeinerungen: F"ur jede multiplikative Teilmenge $S$ von $R$ sei $\Gy_S=\{\ay \subset R\,|\,\ay \cap S \neq \emptyset\}$, f"ur jede nat"urliche Zahl $n\geq 0$ sei $\Gy_D=\{R\} \,\cup\,\{\ay \subsetneq R\,|\,\dim(R/\ay) \leq n\}.$

\begin{lemma}
Ist $R$ ein Integrit"atsring, so gilt:
\begin{enumerate}
\item[(a)]
$\;\Gy_S\,=\,\Gy_1\;\Longleftrightarrow\; R_S $ ist ein K"orper.
\item[(b)]
$\;\Gy_D\,=\,\Gy_1\;\Longleftrightarrow\;\dim(R)=n+1.$
\end{enumerate}
\end{lemma}

\pagebreak 
\begin{proof}
\begin{enumerate}
\item[(a)]
$''\Rightarrow\,''$ \quad
Aus $0\notin \Gy_1$ folgt $0 \notin \Gy_S$, d.h. $0 \notin S,\;R_S \neq 0.$
Jedes $0 \neq \frac{r}{1}\in R_S$ ist aber invertierbar, denn aus $r \neq 0$ folgt $(r) \in \Gy_1,\;(r) \in \Gy_S,\;\lambda r=s_0$ f"ur ein $\lambda \in R,\;s_0 \in S,$ also $\frac{r}{1}\cdot \frac{\lambda}{s_0}=1.$

$''\Leftarrow\,''$ \quad
Aus $\ay \in \Gy_S$, d.h. $0 \neq s_0 \in \ay$
(wegen $0 \notin S$) folgt $\ay \in \Gy_1.$
Aus $\ay \in \Gy_1$, d.h. $0 \neq r_0 \in \ay$ folgt $0 \neq \frac{r_0}{1} \in R_S$,
also nach Voraussetzung $\frac{r_0}{1}\cdot \frac{\lambda}{s_0}=1$ mit $\lambda \in R,\;s_0 \in S$, daraus $r_0\lambda =s_0,\;\ay \in \Gy_S.$
\item[(b)]
$''\Rightarrow\,''$\quad 
Aus $0 \notin \Gy_1,\;0 \notin \Gy_D$, folgt $\dim(R) \geq n+1$, und aus $\my \in \Gy_D,\;\my \in \Gy_1$ folgt
$0 \neq r_0 \in \my,\;(r_0) \in \Gy_1,\;(r_0) \in \Gy_D,\;\dim(R/(r_0))\leq n,$
also nach \cite[Corollary 11.18]{1} $\dim(R)\leq n+1.$

$''\Leftarrow\,''$\quad
$\dim(R) >n$ bedeutet $0 \notin \Gy_D$, f"ur alle $0\neq r \in R$ gilt aber
$\dim(R/(r))\leq n$, also $(r) \in \Gy_D.$
\end{enumerate}
\end{proof}

\begin{corollary}
\begin{enumerate}
\item[(a)]\quad 
$E$ einfach-$\Gy$-teilbar $\;\Longleftrightarrow\; E$ einfach-teilbar und $\Gy =\Gy_1.$
\item[(b)]
\quad $E$ einfach-$\Gy_S$-teilbar $\;\Longleftrightarrow\; E$ einfach-teilbar und $R_S$ ein K"orper.
\item[(c)]\quad
$E$ einfach-$\Gy_D$-teilbar $\;\Longleftrightarrow\; E$ einfach-teilbar und $\dim(R)=n+1$.
\end{enumerate}
\end{corollary}

\begin{proof}
(a)$\; \;''\Leftarrow\,''$ ist klar, und bei $''\Rightarrow\,''$ ist $E$ koprim, also Ann$_R(E)=0$ ein Primideal, $R$ ein Integrit"atsring, $0\notin \Gy,\;\Gy \subset \Gy_1.$
Aus $E \cong R^\circ$ folgt nach (1.7) $R$ einfach-$\Gy$-torsionsfrei, $0 \in \Gy^\ast$ nach (1.1), also $\Gy_1 \subset \Gy.$\\
(b) und (c) folgen jetzt unmittelbar aus dem Lemma.
\end{proof}

F"ur die Gabriel-Topologie $\Gy_S$ l"a"st sich die "Aquivalenz $(i) \leftrightarrow (ii)$ in (3.5) vereinfachen:

\begin{theorem}
Die einfach-$\Gy_S$-teilbaren Untermoduln von $E$ entsprechen genau den maximalen Idealen des Ringes $\hat{R} \otimes_R R_S.$
\end{theorem}

\begin{proof}
Mit der Einbettung $\varphi: R \to \hat{R}$ und $T:= \varphi(S)$ ist die Abbildung $\hat{R} \otimes_R R_S \to \hat{R}_T,\; u \otimes \frac{1}{s} \mapsto u/\varphi(s)$ ein Ringisomorphismus.
Nach (3.5) entsprechen nun die einfach-$\Gy_S$-teilbaren Untermoduln von $E$ genau den Primidealen $P$ von $\hat{R}$, die maximal bez"uglich der Eigenschaft $P \cap R \notin \Gy_S$, d.h. $P \cap T = \emptyset$ sind,
also den maximalen Idealen des Ringes $\hat{R}_T$.
\end{proof}

\section{"Uber die Ideale des Quotientenringes $R_\Gy$}
\setcounter{theorem}{0}

Ist die Gabriel-Topologie $\Gy$ perfekt, so zeigen wir in diesem letzten Abschnitt, da"s die Elemente von $\Gy^\ast$ genau den maximalen Idealen des kommutativen Ringes $R_\Gy$ entsprechen.
Wie in \cite[p.231]{9} hei"se $\Gy$ {\bf perfekt}, wenn $R_\Gy$ als $R$-Modul $\Gy$-teilbar ist.
Aus $\ay \cdot R_\Gy =R_\Gy$ folgt dann umgekehrt $\ay \in \Gy$,
und die kanonische Abbildung $\psi: R \to R_\Gy$ ist ein flacher Ringepimorphismus.

{\bf Beispiel}\quad
Ist $S$ eine multiplikative Teilmenge von $R$ und $R_S=\{\frac{r}{s}\,|\,r \in R,\,s \in S\},$
so ist nach \cite[p.238]{9} $\Gy =\{ \ay \subset R\,|\,\ay \cap S \neq \emptyset\}$ eine perfekte Gabriel-Topologie und $R_\Gy$ ringisomorph zu $R_S$.

\begin{lemma}
Ist $\Gy$ perfekt, so gilt:
\begin{enumerate}
\item[(a)]
F"ur jedes Ideal $\ay$ von $R$ ist $\underline{\psi^{-1}(\ay \cdot R_\Gy)\,=\,\ay^c}$,\\
wobei $\ay^c$ die S"attigung von $\ay$ bez"uglich $\Gy$ sei, d.h. $\ay^c/\ay=T_\Gy(R/\ay)$.
\item[(b)]
F"ur jedes Ideal $A$ von $R_\Gy$ ist $\underline{\psi^{-1}(A) \cdot R_\Gy\,=\,A}.$
\end{enumerate}
\end{lemma}

\begin{proof}
\begin{enumerate}
\item[(a)]\quad
Im kommutativen Diagramm

\vspace{0.5cm}

$$R\quad\stackrel{\psi}{\longrightarrow}\;\; R_\Gy$$
$$\hspace{0.2cm}\nu\;\downarrow\;\;\qquad\qquad \downarrow\;\nu_\Gy$$
$$R/\ay \;\;\stackrel{\psi'}{\longrightarrow} \;\; (R/\ay)_\Gy$$

\vspace{0.5cm}

ist Ke$(\psi') =T_\Gy(R/\ay)=\ay^c/\ay$ und, weil $\Gy$ perfekt ist, Ke$(\nu_\Gy)=\ay\cdot R_\Gy,$
also 
$$\psi^{-1}(\ay \cdot R_\Gy)\;=\;\mbox{Ke}(\nu_\Gy\circ \psi)\;=\;\mbox{Ke}(\psi' \circ \nu)\;=\;\ay^c.$$

\item[(b)]
\quad $''\subset\,''$ ist klar, denn aus $r \in \psi^{-1}(A)$ und $x \in R_\Gy$ folgt $r \cdot x = \psi(r)\,x \in A.$\\
F"ur $''\supset\,''$ sei $x \in A$: Weil $R \stackrel{\psi}{\longrightarrow} R_\Gy$ ein flacher Ringepimorphismus ist,
gibt es nach \cite[p.227]{9} Elemente $r_1,...,r_n \in R$ und $y_1,...,y_n \in R_\Gy$ mit
$$1\;=\; \sum_i \psi(r_i)\;y_i\quad\mbox{ und } \quad x\,\psi(r_i)\;\in \;\mbox{Bi}(\psi)\quad \mbox{ f"ur alle } i,$$
so da"s mit $x\,\psi(r_i)=\psi(t_i)\;$ (alle $t_i \in R$) folgt $ x=\sum\limits_i t_i \cdot y_i,\;t_i \in \psi^{-1}(R_\Gy x)\subset \psi^{-1}(A)$, also
$x \in \psi^{-1}(A) \cdot R_\Gy.$
\end{enumerate}
\end{proof}

\begin{theorem}
Ist $\Gy$ perfekt, so gilt:
\begin{enumerate}
\item[(a)]
Die Ideale von $R_\Gy$ entsprechen genau den $\Gy$-reinen Idealen von $R$.
\item[(b)]
Die Primideale von $R_\Gy$ entsprechen genau den $\Gy$-reinen Primidealen von $R$.
\item[(c)]
Die maximalen Ideale von $R_\Gy$ entsprechen genau den Elementen von $\Gy^\ast$.
\end{enumerate}
\end{theorem}

\begin{proof}
\begin{enumerate}
\item[(a)]
Ein Untermodul $U$ von $M$ hei"st $\Gy${\bf-rein} in $M$, wenn $M/U \;\Gy$-torsionsfrei ist.
In unserem Fall ist $R_\Gy/A$, also auch $R/\psi^{-1}(A)\;\Gy$-torsionsfrei (siehe \cite[Proposition 17.1]{3}), d.h. $\psi^{-1}(A)\;\Gy$-rein in $R$, und mit dem Lemma folgt alles.
\item[(b)]
F"ur jedes Primideal $P \subset R_\Gy$ ist $\psi^{-1}(P)$ ein Primideal in $R$ und wie eben $\psi^{-1}(P)\;\Gy$-rein in $R$,
nach dem Lemma auch $\psi^{-1}(P)\cdot R_\Gy=P$.
Ist umgekehrt $\py$ ein $\Gy$-reines Primideal in $R$, folgt wieder $\psi^{-1}(\py \cdot R_\Gy)=\py$,
und es bleibt zu zeigen, da"s $\py \cdot R_\Gy$ ein Primideal in $R_\Gy$ ist (siehe \cite[Lemma 4]{6}):
$A,B$ Ideale in $R_\Gy$ mit $AB \subset \py\cdot  R_\Gy\,\Rightarrow\,\psi^{-1}(A) \psi^{-1}(B) \subset \py$, etwa $\psi^{-1}(A) \subset \py$, also wie oben $A 
\subset \py \cdot R_\Gy.$

\item[(c)]
F"ur jedes maximale Ideal $P \subset R_\Gy$ ist $\psi^{-1}(P) \in $ Spec$(R)\backslash \Gy$,
und aus $\psi^{-1}(P) \subsetneq \qy$ folgt $P \subsetneq \qy \cdot R_\Gy$ (denn
$P=\qy \cdot R_\Gy$ lieferte den Widerspruch $\psi^{-1}(P)=\qy^c$),
also $\qy \cdot R_\Gy =R_\Gy,\;\qy \in \Gy.$
Ist umgekehrt $\py \in \Gy^\ast$, wird $\py \cdot R_\Gy$ nach (b) ein Primideal in $R_\Gy$,
ja sogar ein maximales Ideal:
Aus $\py \cdot R_\Gy \subsetneq A$ folgt nach (5.1 b) $\py \subsetneq \psi^{-1}(A)$, also $\psi^{-1}(A) \in \Gy,\;A = R_\Gy$.
\end{enumerate}
\end{proof}

\begin{remark}\quad
Ohne Perfektheit kann $|\,$Max$(R_\Gy)|=1$ sein und $\Gy^\ast$ sogar unendlich. Ist z.B. $R$ ein regul"arer Ring mit $\dim(R)\geq 2$,
so ist $\Gy:=\{R\} \cup \{\ay \subsetneq R\,|\,h(\ay)\geq 2\}$ eine Gabriel-Topologie und
$\Gy^\ast=\{\py \in $ Spec$(R)\,|h(\py)=1\}$ unendlich, w"ahrend $R \cong R_\Gy$ ist (siehe \cite[Lemma 3]{6}).
\end{remark}


\end{document}